\documentclass{amsart}[12]
\usepackage{amssymb,lmodern,amsthm}
\usepackage{cleveref}
\newtheorem{example}{Example}
\newtheorem{lem}{Lemma}
\newtheorem{prop}{Proposition}
\newtheorem{thm}{Theorem}

\newcommand\cO{\mathcal{O}}
\newcommand\fp{\mathfrak{p}}
\newcommand\f{\varphi}
\newcommand\g{\psi}
\newcommand\Spec{{\mathrm{Spec\,}}}
\newcommand\ovl{\overline}
\newcommand\nl{\hfill\break}
\newcommand\ds{\displaystyle}

\newcommand\EGA{1}
\newcommand\FG{2}
\newcommand\GW{3}
\newcommand\Gro{4}
\newcommand\Gru{5}
\newcommand\Ka{6}
\newcommand\MG{7}
\newcommand\Ma{8}
\newcommand\MO{9}
\newcommand\RG{10}
\newcommand\St{11}
\newcommand\Va{12}

\begin{document}
\title{Smooth scheme morphisms: a fresh view}
\author{Peter M. Johnson}
\address{Departamento de Matem\'atica, UFPE \\
50740--560, Recife--PE, Brazil}
\email{peterj@dmat.ufpe.br}
\date\today

\begin{abstract}
Relations between some kinds of formal and standard smoothness, for 
morphisms of schemes, are clarified in surprisingly simple and direct ways,
bypassing much of the customarily employed machinery.
Even the deep local-to-global property of formal smoothness has 
a fairly elementary proof, under mild additional hypotheses.
\end{abstract}

\maketitle

\section{Introduction}

Using methods as elementary as possible, accessible to those having only a 
basic understanding of scheme theory, we present efficient new proofs of
the equivalence of some of the main kinds of smoothness of scheme morphisms.  
While smoothness is clearly important in algebraic geometry, there is no general 
agreement about which among several seemingly different definitions should be 
chosen as the basic one that best expresses the concept.  
In certain situations, one approach to smoothness can be distinctly easier to 
work with than others, so it is useful to make transitions whenever convenient.
Our aim is to make such processes thoroughly transparent. 

We innovate by proving key results 
using machinery so limited that concepts such as dimension,
tangent spaces, fibres, regularity and flatness do not appear.  
Discussion of those, as well as the important topic of \'etale morphisms,
has been relegated to another article on consequences and further
characterizations of smoothness.  Only trivial facts will be cited,  
except in many non-essential comments where various concepts 
and results are presupposed in order to make comparisons with other work.
K\"ahler differentials are treated as forming modules rather than sheaves,
and will play an important role in proofs.
Unlike much written in this area, there will be no assumptions or reductions 
involving Noetherianity.
Although they are useful for clarifying some aspects of smoothness, Grothendieck 
topologies other than the Zariski topology will not be treated here.

Stimulus for writing this article came from examining standard sources
selected from a vast literature, notably G\"ortz and Wedhorn [\GW] and 
Vakil [\Va] as well as the encyclopedic Stacks Project book [\St] and EGA [\EGA].  
The first two provide motivation and additional details, important for
elucidating the significance of smoothness.
Except where noted, our definitions coincide with those of [\St].
Some of the most relevant will be repeated below. 
For reasons of stability, references to [\St] are given in the form 
[Tag ....], and can be consulted on-line.
In [\St], as in [\EGA] and many other works (one interesting survey being 
[\Ma]), further variations on notions of smoothness appear, often phrased 
in terms of certain kinds of intersections or using $I$-adic topology.
Homological methods yield deeper results, some presented in [\MG]. 
Our scope is far more limited. 

Formal smoothness and standard smoothness will be regarded as properties of 
morphisms between schemes, to be studied from Zariski-local and stalk-local 
points of view as properties of (homo)morphisms of commutative rings.
The first of these notions of smoothness, due to Grothendieck, has an intrinsic
(presentation-independent) definition well suited for formal diagrammatic 
demonstrations.  The other notion, seemingly more concrete, asserts the 
existence of presentations by generators and relations of a certain form.  
Precise details are given in the next section.  

A different approach, adopted for example in [\St],
takes smoothness to be at heart a property of the sheaf of differentials.
To properly formalize the idea, the naive cotangent complex of a 
morphism [Tag 00S0] is introduced.  This object is shown in [Tag 031J] to be 
in some sense trivial precisely when the morphism is formally smooth.  
An example in [Tag 0635], with a morphism not locally of finite type, 
shows that the full cotangent complex need not behave as well. 
However, smooth morphisms are by definition required to be locally finitely 
presented.  In that case, a refined approach to differentials allows  
formal smoothness to be exploited in a previously unknown way,
yielding a proof of the first theorem below. 
It then becomes easy to bypass the naive cotangent complex in proofs of results 
about smoothness, if one is willing to take the foundation stone of the theory
to be Grothendieck's abstract definition rather than one based explicitly on differentials.

In what follows, the non-standard term `around $x$' pertains to local properties,
those that hold on some restriction $U \to V$ to affine open subschemes with 
$x \in U$, and on further such restrictions.
We also write `at $x$' for properties that are are even more local, defined from a 
single map $\cO_{Y,\f(x)} \to \cO_{X,x}$ to the stalk at $x$ from the stalk at $\f(x)$.
Such maps are local homomorphisms $R_{\fp'} \to S_\fp$ of local rings.
In the terminology just introduced, and with the definitions of smoothness 
as given in the next section, the first fundamental result is:

\begin{thm}\label{fund}
Let $\f: X \to Y$ be a scheme morphism, locally of finite presentation.
For each point $x$ of $X$, the following are equivalent:
\nl (a) $\f$ is formally smooth around $x$;
\nl (b) $\f$ is standard smooth around $x$;
\nl (c) $\f$ is formally smooth at $x$;
\nl (d) $\f$ is standard smooth at $x$.

\end{thm}

The $x \in X$ satisfying these conditions form the point set of
an open subscheme of $X$, the {\em smooth locus} of $\f$.
When this is $X$, $\f$ is said to be smooth.
As a byproduct of the proof techniques, the idea of studying smoothness
from given generators and relations can be justified.  It will be clear
after the proof of \Cref{fund} that it is straightforward to test at stalks. 
We even obtain a formula, too cumbersome to be of practical use, for the 
smooth locus.  Some open subsets of $\Spec(S)$ needed just below will be 
principal open sets $D(f)$, $f \in S$, while others are complements of 
closed sets $V(I)$, $I$ an ideal of $S$.

\begin{thm}\label{locus} If $S \cong R[x_1,\dots x_n]/(f_1, \dots, f_c)$, the 
smooth locus of $\f: \Spec(S)$ $\to \Spec(R)$ is a union of $\binom{n+c}c$ 
open sets $D(\Delta) \cap V(I)^\mathsf{c}$.  Each $\Delta$, a Jacobian determinant,
and $I$, an ideal of the form $(I_0:I_1)$, is defined by an explicit formula 
involving the given $f_i \in R[x_1,\dots x_n]$ and their formal partial derivatives
$\frac{\partial f_i} {\partial x_j}$.
\end{thm}

The final section concerns a distinctly deeper fundamental result: an
arbitrary scheme morphism $\f:X \to Y$ is formally smooth precisely when 
it is locally formally smooth, by which we mean this holds around all 
points of $X$, with `around' as defined above.
Under an assumption that includes the case where $\f$ is locally of 
finite type, we show that elementary methods, carefully deployed, 
yield a proof technically simpler than previously known ones.

\section{Definitions and notation}

All rings considered will be commutative algebras (with $1$, and
allowing the case $1=0$) over a ring called $R$.  Maps between rings,
including derivations, are taken to be $R$-linear.  As in [Tag 00TH],
a ring morphism $R \to S$ is {\em formally smooth} if, in the category
of $R$-algebras, every morphism $S \to A/I$, where $I$ is an ideal of
$A$ with $I^2 = 0$, lifts to (or factors through) a morphism $S \to A$.
In other words, the defining property is that all these $A \to A/I$
induce surjective functions $\textrm{Hom}_R(S,A)\to\textrm{Hom}_R(S,A/I)$.

A simple observation, not needed below, motivates the definition.  On 
dropping the assumption $I^2 = 0$, when $S$ is formally smooth over $R$
there is instead a lift $S \to A/I^2$, then in turn a lift $S \to A/I^4$,
and so on.  One can easily interpose $A/I^3$ and other terms.  Thus,
after countably many choices, the given $S \to A/I$ factors through $S
\to \hat A$, where $\hat{A}$ is the inverse limit
$\ds\lim_{\longleftarrow} A/I^n\!,$ the $I$-adic completion of $A$.

Formal smoothness of a scheme morphism $\f: X \to Y$ has an analogous
definition [Tag 02GZ], best stated relative to the category of schemes
over $Y$ [Tag 01JX].   We say $\f: X \to Y$ (or even just $X$) is 
{\em formally smooth} if every morphism $\Spec(A/I) \to \Spec(A)$ over $Y$,
where $I^2 = 0$, induces a surjective function 
$\textrm{Hom}_Y(\Spec(A),X) \to \textrm{Hom}_Y(\Spec(A/I),X)$. 

Among various results that are trivial consequences of the definition,
we note that of [Tag 02H3]: every restriction of $\f:X \to Y$ 
to a morphism $U \to V$ between open subschemes of $X$ and $Y$ is 
formally smooth if $\f$ is.  It is also clear that any such restriction
$U \to V$ is formally smooth precisely when $U \to Y$ is.
Further properties, not needed here, show the robustness of this notion.
The really striking fact, mentioned earlier, is that the formal smoothness 
of a morphism turns out to be determined by local properties alone
(it is local in the sense of [Tag 01SS]).  One can find further details of
interest, in the context of algebraic spaces, in [Tag 049R].

Given an explicit finite presentation of $S$ over $R$, say 
$S \cong P/(f_1,\dots,f_c)$, where $P$ denotes the polynomial ring 
$P = R[x_1,\dots x_n]$, the image in $S$ of any $f$ in $P$ is also 
called $f$ when the meaning is clear.  We prefer to use the name $I_0$ 
for the ideal $(f_1,\dots,f_c)$ of $S$, reserving $I$ for other uses. 
The cases $c=0$ or $n=0$ are allowed, rewording where necessary. 

Departing slightly from the presentation-dependent definition in [Tag 00T6], 
we say a ring morphism $R \to S$, or $S$ as an $R$-algebra, is {\em standard
smooth} over $R$ if $S$ can be presented as above, where $c \leq n$
and the $c \times c$ matrix $J$ with $(i,j)$-entry $\frac{\partial f_i}
{\partial x_j}$ maps to an invertible matrix over $S$, or equivalently
$\textrm{det}(J)$ is invertible in $S$.  Also say $R \to S$ is 
{\em essentially standard smooth} if $S$ is isomorphic to a localization
of some standard smooth $R$-algebra.  This is of most interest when
$S$ arises from localizing at a prime ideal.  At the other extreme,
a principal localization $S = T_g$ $(g \in T)$ of a standard smooth
$R$-algebra $T$ remains standard smooth, from the now commonplace
Rabinowitsch trick: introduce a new variable $x$ and relator $x.g-1$.

A morphism $\f: X \to Y$ of schemes is {\em standard smooth around a point} 
$x \in X$ if there are affine open subsets $U$, $V$ with $x \in U$
such that the restriction of $\f$ to $U \to V$ is induced from a standard
smooth ring homomorphism.  By definition $\{x \in X \mid \f\, \textrm{
is standard smooth around } x\}$ is open in $X$.  On this, the smooth
locus, $\f$ is said to be {\em locally standard smooth}, a property
with many desirable consequences.  
Maps between stalks cannot be expected to be finitely presented. 
Thus, when $\f$ is said to be {\em standard smooth at $x$}, for some $x \in X$, 
we mean that the 
induced map $\cO_{Y,\f(x)} \to \cO_{X,x}$ is essentially standard smooth. 
Here $\cO_{Y,\f(x)}$ plays the role of $R$.

\begin{example} Suppose $2$ is a unit of $R$ and $S = R[x,y]/(x^2+y^2-1)$.
While $R \to S$ is not standard smooth (which could be verified later, from
the module of differentials), the morphism $\Spec(S) \to \Spec(R)$ is locally 
standard smooth. This follows at once, using the open cover $\{D(x), D(y)\}$ 
of $\Spec(S)$.
\end{example}

\section{Derivations and differentials}

The usual differential on $P = R[x_1, \dots, x_n]$ is $df = \sum_j
\frac{\partial f} {\partial x_j}\,dx_j$, where to be definite the $dx_j$
(in order) are taken to be the elements of the canonical basis of the module $P^n$.
With $S \cong P/I_0$ and $I_0 = (f_1, \dots, f_c)$ as before, there is a
derivation $d = d_S: S \to M(S)$, where $M(S)$, the $S$-module of differentials,
is the quotient of $S^n$ by the submodule generated by $\{df \mid f \in I_0\}$.
It suffices to impose the relations $df_i = 0$ ($1 \leq i \leq c$), since all
$f_i$ act as 0 on $M(S)$, so $d(\sum_i g_i f_i) = \sum_i (g_i.df_i + f_i.dg_i) = 0$.
However, finiteness of the number of generators and relations is clearly 
non-essential here, and indeed throughout this whole section.

Since $M(S)$ is a concrete representative of the K\"ahler module [Tag 07BK], 
satisfying (with $d_S$) the universal property [Tag 00R0] of being the 
freest possible $S$-module of $R$-differentials of $S$, we adopt the standard
notation $\Omega_{S/R}$ in place of $M(S)$.  A different representative, 
visibly presentation-independent, is furnished in [Tag 00RN].
A more refined notion of universality will appear in the next lemma.

The most subtle idea of the whole development is to compare certain
derivations, as in [Tag 02HP] and [Tag 031I] (which includes a converse).
These leave a few important details for readers to check.  We follow
roughly similar lines, placing more emphasis on universal concepts,
proving exactly what is necessary for our purposes.

To set notation, given an ideal $I$ of an $R$-algebra $S$, write
$\ovl S = S/I$, $S' = S/I^2$, with $s \mapsto s' \mapsto \ovl s$ under the
maps $S \to S' \to \ovl S$.  Note that $\ovl S$-modules can be regarded as
$S$-modules annihilated by $I$, or as $S'$-modules annihilated by $I/I^2$.
Similar remarks apply to derivations $d$.  The universal property in the
next result should be clear from the way it will be used.
As always, $R$-linearity is implicit.

\begin{lem}\label{univ} 
The usual $d_S: S \to \Omega = \Omega_{S/R}$, universal for derivations from
$S$ to $S$-modules, induces a $d_S': S' \to \ovl \Omega = \Omega/I\Omega$
that is universal for derivations from $S'$ to $\ovl S$-modules,
where $S' = S/I^2$ and $\ovl S = S/I$. 
\end{lem}

\begin{proof}
Let $M$ be an arbitrary $\ovl S$-module, so $I.M = 0$ for $M$ as an
$S$-module. Every $d: S \to M$ must satisfy $d(I^2) = 0$, so induces
some $d': S' \to M$.  This correspondence $d \mapsto d'$ is bijective,
since any $d': S' \to M$ lifts via $S \to S'$ to some $d: S \to M$ that
in turn induces $d'$.  The bijection has a naturality property: given 
an $\ovl S$-module map $\alpha: M \to M'$, for every $d: S \to M$ we 
have $(\alpha \circ d)' = \alpha \circ d'$.

The derivation called $d_S': S' \to \ovl \Omega$ is the one induced from 
the composition of the universal $d_S: S \to \Omega$ and $\Omega \to 
\ovl \Omega$.  To verify the claimed universal property of $d_S'$,
start with any $d': S' \to M$, where $I.M=0$.  This is induced from some
$d:S \to M$, which factors through $d_S$ as $S \to \Omega \to \ovl \Omega
\to M$ for some unique $\ovl S$-linear map $\alpha: \ovl \Omega \to M$.
By naturality, uniqueness also holds for the factorization
$d': S' \to \ovl \Omega \to M$ through $d_S'$. 
\end{proof}

\begin{prop}\label{summ}
With notation as above, assume that $R \to \ovl S$ is formally smooth.
Then $d_S': S' \to \ovl \Omega$ restricts to an isomorphism of
$\ovl S$-modules between $I/I^2$ and a direct summand of $\ovl \Omega$,  
giving a split exact sequence 
$0 \to I/I^2 \to \ovl\Omega \to \Omega_{\ovl S/R} \to 0$.
In addition, $\Omega_{\ovl S/R}$ is a projective $\ovl S$-module.
\end{prop}

\begin{proof}
Since $\ovl S$ is formally smooth (over $R$) and isomorphic to $S'/(I/I^2)$, 
where $I/I^2$ has square zero, some injective ring map $\beta: \ovl S \to S'$ 
is a right inverse of $S' \to \ovl S$.  Thus $S'$ is a split extension 
$(I/I^2) + \beta(\ovl S)$, a direct sum as $R$-modules.  Projection 
yields an $R$-linear function $D: S' \to I/I^2$ that satisfies
$s' = D(s')+\beta(\ovl s)$  $(s' \in S')$. 
A calculation of $s'.t'$ in $S'$, using $D(s').D(t') = 0$, shows that
$D$ is a derivation:
\[D(s'.t') = s'.D(t') + t'.D(s') \quad (s',t' \in S').\]

By \Cref{univ}, $D$ factors as $d_S': S' \to \ovl \Omega$ followed by 
some $\ovl S$-linear $\gamma: \ovl \Omega \to I/I^2$.  However $D$, being
a projection map, restricts to the identity function on $I/I^2$.  
Thus $d'_S$ maps $I/I^2$ isomorphically to an $\ovl S$-submodule of 
$\ovl \Omega$, with ker$(\gamma)$ as a complementary direct summand.
We have ker$(\gamma) \cong \ovl\Omega/d_S'(I/I^2) \cong \Omega_{\ovl S/R}$.

To conclude, $\ovl S$ can also be identified with a quotient $\ovl P = P/I_0$, 
where $P$ is a polynomial ring over $R$ or (for later use) a localization of some 
such ring.  Then $\Omega_{\ovl S/R} \cong \Omega_{\ovl P/R}$.  In the new 
split exact sequence, $\ovl \Omega$ is $\Omega_{P/R}/I_0\Omega_{P/R}$, which is a 
free $\ovl P$-module, using the following observation on localizations.
\end{proof}

Differentials behave well under localization [Tag 00RT(2)].
In brief, if $S \cong P/I_0$ is localized at a prime
ideal $\fp$,  with corresponding $\fp' \in \Spec(R)$, $\tilde\fp
\in \Spec(P)$,  there is a well-defined $d: P_{\tilde\fp} \to
P_{\tilde\fp} \otimes_P \Omega_{P/R}$, $f/g \mapsto (1/g)df - (f/g^2)dg$,
which induces $d: S_\fp \to S_\fp \otimes_S \Omega_{S/R}$, clearly
universal for $S_\fp$.  The last module can be identified with
$\Omega_{S_\fp/R_{\fp'}}$, as $R$ acts through  $R_{\fp'}$.
Similar results hold for any localization of $S$.

\section{Local relations between formal and standard smoothness}

Arguments in this section rely crucially on finiteness of presentation.
We start with the well-known fact that formal smoothness holds for presentations 
satisfying a version of the Jacobian criterion. 
The usual proof, given below, uses ideas from deformation theory closely related to
iterative methods for approximating solutions of systems of equations.
As often occurs, a desired result is obtained on restricting to an open subscheme, 
here $\Spec(S_\Delta)$ within $\Spec(S)$.  Renumberings of $x_1, \dots, x_n$ may 
alter $\Delta$, producing various subschemes, some possibly empty.

\begin{prop}\label{loc}
Suppose $S = P/I_0$, where $P = R[x_1,\dots,x_n]$, $I_0=(f_1, \dots, f_c)$
and $c \leq n$.  Let $J$ be the $c \times c$ matrix with $(i,j)$-entry
$\frac{\partial f_i}{\partial x_j} \in P$, with $\Delta$ the image of
$\textrm{det}(J)$ in $S$.  Then the principal localization $S_\Delta$
is a formally smooth $R$-algebra. 
\end{prop}

\begin{proof}
One can obtain a presentation for $S_\Delta$ from that of $S$ by adding
a generator $x_{n+1}$ and relator $x_{n+1}.\Delta - 1$. This produces a
new matrix whose determinant is the square of the original one.  Thus,
changing notation, it can and will be assumed that $\Delta$ is already
a unit (invertible in $S$).

Given any map from $S = P/I_0$ to $\bar A=A/Z$, where $Z$ is an ideal of
square zero, for each generator $x_j$ of $P$ choose $a_j \in A$ so that 
$\ovl{x\,}\!_j \mapsto \ovl{a\,}\!_j$ $(1 \leq j \leq n)$.  This defines
a homomorphism $P \to A$, $x_j \mapsto a_j$.  It will be deformed to one
annihilating $I_0$, with $x_j \mapsto a_j+z_j$, for certain $z_j \in Z$. 
The given map $S \to \bar A$ will then lift to $S \to A$.

Since $Z^2 = 0$, each condition that $f_i$ maps to $0$ is a linear
(or affine) equation in the $z_j$.  Explicitly, in vector notation,
$f_i(\vec{a}+\vec{z}) = f_i(\vec{a}) + \sum_j \frac{\partial f_i}{\partial x_j}
(\vec{a}).z_j = 0$.  All $z_j$ with $j > c$ can be freely chosen, say set to $0$.  
The coefficient matrix is then an image of $J$, defined above.  The system can 
be solved because its determinant has the same image in $\bar A$ as 
$\Delta = \textrm{det}(J)$, hence is a unit of $A$ since $Z$ is nilpotent. 
\end{proof}

To prepare for use in arguments, we refine previous notation.
Given ideals $I_0 \subseteq I_1$ of $P = R[x_1,\dots,x_n]$, 
define $S = P/I_0$, $\ovl S = P/I_1$, $I = I_1/I_0$. 
Then $I$ is an ideal of $S$ for which $S/I \cong \ovl S$.
To fix notation, write $I_0 = (f_1, \dots, f_{c_0})$ (so $c_0 = c$)
and  $I_1 = (f_1, \dots, f_{c_1})$.
In practice, generators for $I_1$ will be given and $I_0$ will be
defined by selecting a certain number $c_0$ of these.  There are Jacobian
matrices $J_0$ and $J_1$, of sizes $c_0 \times n$, and $c_1 \times n$,
each with $(i,j)$-entry the image of $\frac{\partial f_i}{\partial x_j}$
in $S$.  Fix $\ovl\fp \in \Spec(\ovl S)$, lifting it to $\fp \in \Spec(S)$
and $\tilde\fp \in \Spec(P)$.  Let $\ovl J_0$, $\ovl J_1$ denote the images of
$J_0$, $J_1$ with entries in $\kappa(\ovl\fp)$, the field onto which
$\ovl S_{\ovl\fp}$ maps.

\begin{prop}\label{I1I0}
With notation as just above, let $c_0$ be the rank of $\ovl J_1$ over
$\kappa(\ovl\fp)$.  After some reordering of the $f_i$, the rows of $\ovl J_0$,
corresponding to generators $f_1, \dots f_{c_0}$ of $I_0$, form a basis
for the rowspace of $\ovl J_1$.  The universal derivation $d_S: S \to \Omega$,
where $S = P/I_0$, maps the ideal $I = I_1/I_0$ into $\fp \Omega$. 
\end{prop}

\begin{proof} 
The rows of $\ovl J_1$ lie in $k(\ovl\fp)^n$, so contain a maximal independent
set of size $c_0 \leq n$.  We work with polynomials, and derivations
induced from $d: P\to P^n$, $f\mapsto(\frac{\partial f}{\partial x_1},
\dots \frac{\partial f}{\partial x_n})$, where $P = R[x_1,\dots,x_n]$.
Let $i$ range over $[1,c_0]$.  By the choice of generators $f_i$ of $I_0$,
for each $f \in I_1$ there are $g_i \in P$ for which $df - \sum_i g_i.df_i
\in \tilde\fp^n \subset P^n$.  Recall that $S = P/I_0$, so the $df_i$ are 
relators of the universal $S$-module $\Omega = \Omega_{S/R}$, a quotient of 
$P^n$ (via $S^n)$.  Thus $d_S: S \to \Omega$ maps $I_1/I_0$ into $\fp \Omega$.
\end{proof}

The following basic result will also be needed.

\begin{lem}\label{Naka}
Let $M$ be a direct sum $M_1 \oplus M_2$ of modules over $S_\fp$, 
where $\fp \in \Spec(S)$.  If the submodule $M_1$ is finitely generated 
and contained in $\fp M$, then $M_1 = 0$. 
\end{lem}

\begin{proof}
Note that $\ovl{M} = M/\fp M$ is a direct sum $\ovl{M}_1\oplus\ovl{M}_2$,
as $\fp M$ is also a direct sum.  Then $\fp M_1 = M_1$ since $\ovl{M}_1=0$.
By Nakayama's Lemma [Tag 00DV], $M_1 = 0$. 
\end{proof}

\begin{proof}[Proof of \Cref{fund}]
All parts reduce to assertions about ring morphisms.

(b) $\Rightarrow$ (a): This was done in \Cref{loc}.

(a) $\Rightarrow$ (c): Easy arguments with diagrams involving $A \to A/I$,
where $I^2 = 0$, show that formal smoothness of $R \to S$ is preserved
under any localization of $S$, say $S_\fp$, then also for the induced
$R_{\fp'} \to S_\fp$, where $\fp \mapsto \fp'$ in $\Spec(S) \to \Spec(R)$.

(b) $\Rightarrow$ (d): This is trivial, from the definitions.

(d) $\Rightarrow$ (c): If the local ring $S_\fp$ is the localization
of a standard smooth $R$-algebra $S_1$, by previous steps $R \to S_1$
is formally smooth and, by localizing, so is $R_{\fp'} \to S_\fp$.

(c) $\Rightarrow$ (b): To retain consistency with earlier notation, we
begin with an $R$-algebra called $\ovl S$, finitely presented as $P/I_1$,
and a formally smooth stalk map $R_{\fp'} \to \ovl S_{\ovl\fp}$.  It suffices 
to find some $\ovl g \in \ovl S \backslash \ovl\fp$ for which $\ovl S_{\ovl g}$ 
is standard smooth over $R$.

After choosing an ideal $I_0 \subset I_1$ of $P$ as in \Cref{I1I0},
write as before $S = P/I_0$, $I = I_1/I_0$, $S' = S/I^2$,
and note that $\ovl S \cong S/I$.
By \Cref{I1I0}, $d_S: S \to \Omega$ maps $I$ into $\fp\Omega$ and so, from 
\Cref{univ}, $d_S': S' \to \ovl\Omega$ maps $I/I^2$ into $\ovl\fp\,\ovl\Omega$.
 
To localize, lift the given $\ovl\fp$ to $\tilde\fp \in \Spec(P)$ and
work with analogous definitions from $P_{\tilde\fp}$ in place of $P$.
Let for example $(I_1)$ denote $I_1.P_{\tilde\fp}$ and $(I)$ be $(I_1)/(I_0)$.
To simplify notation, $T$ in place of $S$ signifies localization.
Thus the given $\ovl S_{\ovl\fp}$ becomes $\ovl T$.
We rename $\Omega$ as $\Omega_{T/R_{\fp'}}$ (isomorphic to 
$S_\fp\otimes_S \Omega_{S/R}$), and let $\ovl\Omega = \Omega/(I)\Omega$.

The result involving $d_S'$ now yields: $d_T': T' \to \ovl\Omega$ maps 
$(I)/(I)^2$ into $\ovl\fp\ovl\Omega$.  By hypothesis $\ovl T \cong T/(I)$ 
is formally smooth over $R_{\fp'}$, so \Cref{summ} implies that 
$(I)/(I)^2$ maps isomorphically onto a direct summand of $\ovl\Omega$.
By \Cref{Naka}, $(I) = (I).(I)$.
As $(I) = (I_1)/(I_0)$ is a finitely-generated proper ideal of the 
local ring $T = P_{\tilde\fp}/(I_0)$, a form of Nakayama's Lemma implies 
that $(I)$ is the zero ideal, so $(I_1)= (I_0)$ in $P_{\tilde\fp}$. 
Returning to ideals of $P$, finite generation of $I_1$ now implies that 
$gI_1 \subset I_0$ for some $g \in P \backslash \tilde\fp$.  Thus $\ovl 
S_{\ovl g} \cong P_g/I_0P_g$, a standard smooth $R$-algebra by \Cref{loc}.
\end{proof}

Recall that \Cref{I1I0} shows how to choose a certain number $c_0$ of
polynomials from any given list of generators of $I_1$.
Assuming a formally smooth stalk map, it was shown just above that all
remaining relators in the list become redundant, after passing to a
suitable principal localization.  As a rough restatement, within each
stalk a failure of redundancy in the additional relators is equivalent
to the failure of smoothness for the map at that stalk.

\section{A formula for the smooth locus}

Going beyond a test for the smoothness of individual stalk maps, we present
a formula for the smooth locus of an open affine subscheme over $R$, now
written as $\Spec(\ovl S)$, $S \cong P/I_1$, to conform with earlier conventions. 
Here $P = R[x_1,\dots,x_n]$ and $I_1 = (f_1, \dots, f_{c})$, so $c=c_1$.
The idea is to study square submatrices $J$ of the $c \times n$ Jacobian 
matrix defined from the given presentation of $\ovl S$.

\begin{proof}[Proof of \Cref{locus}]
Given any $\ovl\fp \in \Spec(\ovl S)$, let $\fp$ be the lift to $\Spec(P)$.
Following \Cref{I1I0}, there is a largest $c_0 \geq 0$ for which
one can choose $c_0$ of the $f_i$, generating an ideal $I_0$, then
$c_0$ of the variables $x_j$, used to form a submatrix $J$ of the
Jacobian matrix such that $\textrm{det}(J) \notin \fp$.
If $\ovl\fp$ is in the smooth locus of $\ovl S$, we know from the end of 
the proof of \Cref{fund} that $gI_1 \subset I_0$ for some $g \notin \fp$, or 
equivalently $\fp \notin V((I_0:I_1))$, where $(I_0:I_1) = 
\{g \in P \mid gI_1 \subset I_0\}$.  Conversely, if $\fp$ satisfies this 
last condition then $\ovl\fp$ is in the smooth locus, by \Cref{loc} and
preservation of smoothness (of either kind) under localization at prime ideals.

Reversing the previous point of view, one now starts by choosing a natural number
$c_0 \leq \textrm{min}(c,n)$ and selecting $c_0$ members from each of the lists 
$f_1, \dots, f_c$ and $x_1, \dots, x_n$.  From these one defines an ideal
$I_0 \subset I_1$ and a $c_0 \times c_0$ submatrix $J$ of the Jacobian matrix.
Let $\Delta$ be the image in $S$ of $\textrm{det}(J)$.  As just above, each
$\ovl\fp$ in $D(\Delta) \cap V((I_0:I_1))^\mathsf{c}$ lies in the smooth locus 
of $\ovl S$. Conversely, varying the choices, such open subsets of $\Spec(\ovl S)$
cover the whole smooth locus, by arguments in the previous paragraph.

The number of open sets used is the number of square submatrices
(including the empty matrix) of a general $c \times n$ matrix.
This is symmetric in $c$ and $n$, so to count we may suppose $c \leq n$.  
The number is $\sum_{i=0}^c {\binom{n}i}{\binom{c}i} =
  \sum_{i=0}^c {\binom{n}i}{\binom c{c-i}} = {\binom{n+c}c}.$
\end{proof}

\section{Local to global for formally smooth morphisms}

The nontrival direction of the following fundamental theorem is
a remarkable local-to-global property of formal smoothness.
It is basically a cohomological result from deformation theory, but
we will show how it follows from a careful analysis using elementary methods.
Instead of a local hypothesis of the form `$R \to S$ of finite type', it
suffices to require only `essentially of finite type' [Tag 00QM]:
$S$ is a localization of an $R$-algebra of finite type.
Examples of such morphisms, not locally finite, are easily found
using transcendental extensions of fields.
With `around' in the sense defined in the introduction, the result we prove is:

\begin{thm}\label{ffund}
For every scheme morphism $\f: X \to Y\;$that is locally essentially of finite type,
the following are equivalent:
\nl (a) $\f$ is formally smooth;
\nl (b) $\f$ is formally smooth around all points of $X$.
\end{thm}

The equivalence is in fact true for arbitrary $\f$, but the given hypothesis on $\f$ is 
sufficiently general to include all situations that normally occur, and significantly 
simplifies the proof.  In this context it is usual to cite a result on the triviality 
of \v Cech cohomology for quasicoherent sheaves on affine schemes, but we do not do so. 
 
Some historical background is given for those interested. 
The very last (fourth) issue of [\EGA], Vol.~4 states the equivalence in full 
generality as Proposition (17.1.6), but the proof cites 
something that requires $\f$ to be locally of finite presentation. 
Grothendieck [\Gro, Remarque 9.5.8] (also see [\MO]) discussed that gap,
found soon after publication.  He stated that there is no problem for 
locally finite morphisms, adding that the general
result would follow from a proof of a conjecture on 
the structure of locally projective modules.  A few years later, in 1971, 
the deep analysis carried out by Raynaud and Gruson [\RG], building on 
earlier work of Kaplansky [\Ka], implied a full solution, once some 
technical details had been corrected in [\Gru].  

We ignore all this, and later related work, as the case we treat can be 
handled using only rudimentary machinery.  
This provides an instructive contrast to more conventional methods.
At the end, adjustments are made via partitions of unity 
analogous to ones used in complex analysis to prove certain theorems of 
Dolbeault and de Rham---see for example [\FG, p.\ 311].

Just as in [Tab 01UP], a one-line calculation (omitted) yields:

\begin{lem}\label{torsor}
Let $\bar A = A/I$, where $I^2 = 0$.  Given a ring morphism
$\sigma: R \to A$, regard $I$ as an $R$-module via the induced map 
$\ovl\sigma:R\to\bar A$ composed with the natural action of $\bar A$ on $I$.
Then the morphisms $R \to A$ that induce $\ovl\sigma$ 
are the functions of the form $\sigma+\delta$
for which $\delta: R \to I$ is an $R$-derivation. 
\end{lem}

\begin{proof}[Proof of \Cref{ffund}]
One direction is trivial, from observations made in Section~2.
Given (b), one sees from the hypotheses
that for each $x \in X$  
there are affine sets $U \subset X$ and $V \subset Y$ with $x \in U$
such that $f$ restricts to a formally smooth morphism $U \to V$ that is 
essentially of finite type.
Using notation $\bar A = A/I$, where $I^2 = 0$, and $T' = \Spec(\bar A)$,
with underlying set identified throughout with that of $T = \Spec(A)$,
it must be shown that any morphism $\g: T' \to X$ factors through the 
usual $T' \to T$.   The morphisms here are already over $Y$, via 
compositions with $\f: X \to Y$.  

As a first step, in which notation will be defined and various choices made, each 
$t \in T'$ lies in a principal affine open set contained in some $\g^{-1}(U)$ 
with $U$ as above. The affine scheme $T'$ is covered by finitely many of these
sets, henceforth called $T'_1, \dots T'_l$, of the form $\Spec(\ovl{A_{f_i}})$. 
We write the localization $A_{f_i}$ as $A_i$ and use a similar 
convention for the $\bar A$-module $I$, but not more generally. 
Variables $i,j,k$ range over $[1..l]$, the initial focus being on a single $i$.
Also choose as above affine opens $U_i \subset X$, 
$V_i \subset Y$ so that $\g$ and $\f$ restrict to maps $T'_i \to U_i \to V_i$.
Each corresponding ring homomorphism $R_i \to S_i \to \bar A_i$ is regarded as 
forming a morphism $S_i \to \bar A_i$ of $R_i$-algebras, with $S_i$ 
essentially finite over $R_i$. 
By the formal smoothness assumption, one can choose a lift $\sigma_i: S_i \to A_{i}$.  

To study morphisms on subschemes of $T = \Spec(A)$,
notation will be extended in an obvious way and natural identifications made, 
so for example $T_{ij}$ denotes the scheme $\Spec(A_{ij})$ on the set 
$T_i \cap T_j$.  The case $j=i$ is not excluded but will be of no interest.
The morphism $T_{ij} \to U_i$, obtained from composing $\sigma_i$ 
with $A_i \to A_{ij}$, factors through $U_{ij} = U_i \cap U_j$, a scheme 
which need not be affine but is open in $U_i$ and in $U_j$.  Thus both 
$\sigma_i$ and $\sigma_j$ will induce scheme morphisms $T_{ij} \to U_{ij}$. 
Composing with the inclusion $U_{ij} \to U_i$ produces two 
lifts $S_i \to A_{ij}$ of the same morphism $S_i \to \bar A_i \to \bar A_{ij}$.
By \Cref{torsor}, their difference (in the order $i,j$) is an
$R_i$-linear derivation $\delta_{ij}: S_i \to I_{ij}$.
Recall that $\delta_{ij}$ then factors as $S_i \to \Omega_{S_i/R_i} \to I_{ij}$.
  
By assumption $S_i$ is a localization of the $R_i$-subalgebra generated by 
some finite set $\{s_\gamma \mid \gamma \in \Gamma_i\}$ of its elements.
Each $\delta_{ij}(s_\gamma)$ is of the form $z_{\gamma,j}/f_j^{N(\gamma,j)}$,
for some chosen $z_{\gamma,j} \in I_i$ that minimizes 
$N(\gamma,j) \in \mathbb{N}$, with $i$ implicitly determined by the index $\gamma$.
We can then choose some upper bound $N \in \mathbb{N}$ 
for all the $N(\gamma,j)$, letting $i$ also vary.
 
 From the end of the proof of \Cref{summ}, $\Omega_{S_i/R_i}$ can be identified
with a direct summand of $\Omega_{P'_i/R_i}$, where $P'_i$ is a suitable
localization of the polynomial algebra with free generators $x_\gamma\ 
(\gamma \in \Gamma_i)$, such that $x_\gamma \mapsto s_\gamma$ defines a 
surjective morphism $P'_i \to S_i$. Passing to differentials, the $dx_\gamma$ 
form a $P'$-basis of $\Omega_{P'_i/R_i}$.  The projection $dx_\gamma \mapsto 
ds_\gamma$ fixes $\Omega_{S_i/R_i}$, and the above factorization of 
$\delta_{ij}$ gives a map with $ds_\gamma \mapsto z_{\gamma,j}/f_j^{N(\gamma,j)}$.
A multiple of the composition lifts to a map $\Omega_{P'_i/R_i} \to I_i$,
$dx_\gamma \mapsto f_j^{N-N(\gamma,j)}z_{\gamma,j}$.
Restriction to $\Omega_{S_i/R_i}$ then yields an $R_i$-linear derivation 
$\delta^{(N)}_{ij}: S_i \to \Omega_{S_i/R_i} \to I_i$ which, composed
with $I_i \to I_{ij}$, is $f_j^N\delta_{ij}: S_i \to I_{ij}$.  
As $N$ may later be increased further, note that for $M \in \mathbb{N}$ we have
$\delta^{(M+N)}_{ij} = f_j^M\delta^{(N)}_{ij}\!.$ 

In much the same way that $\delta_{ij}$ was defined, one can find a derivation 
$S_i \to I_{ij}$ that is the difference of two homomorphisms $S_i \to A_{ij}$ 
and, when calculated on the stalks of $U_{ij}$, which are algebras
over both $R_i$ and $R_j$, has the formula
$\delta_{ik}^{(N)}-\delta_{jk}^{(N)} - f_k^N\delta_{ij}$. 
 From the definitions, the image of each element of $S_i$ under this 
derivation lies in the kernel of $I_{ij} \to I_{ijk}$, so
is annihilated in $I_{ij}$ by some power of $f_k$.  
The derivation is determined by its effect on the finite set 
$\{s_\gamma \mid \gamma \in \Gamma_i\}$, and is multiplied by  $f_k^M$
if $N$ is increased by $M$.  Thus, for a sufficiently large $N$, all these
derivations $S_i \to I_{ij}$, with $i,j,k$ now varying over $1..l$, are 
identically zero.  One then has a formula, valid on stalks of $U_{ij}$:
\[ \delta_{ik}^{(N)}-\delta_{jk}^{(N)} = f_k^N\delta_{ij}.\]
The sets $D(f_k)$ cover $\Spec(A)$, so there are 
$h_k \in A$ with $\sum_k h_k f_k^N =1$.  For each $i$,
define the ring morphism $\sigma'_i : S_i \to A_i$, a
new lift of the original map $S_i \to \bar A_i$,
by $\sigma'_i = \sigma_i - \sum_k h_k\delta_{ik}^{(N)}$.
To compare corresponding scheme morphisms $T_i \to U_i$ and $T_j \to U_j$
on the overlap $T_{ij} \to U_{ij}$, one can take differences of
the ring homomorphisms on stalks (or use suitable affine subschemes
of $U_{ij}$).  Using previous formulas, calculations taking values in stalks 
of $T_{ij}$, then later just in $A_{ij}$, give:  \[ \sigma'_i - \sigma'_j = 
(\sigma_i - \sigma_j) - \sum_k h_k f_k^N \delta_{ij} = \delta_{ij}-\delta_{ij} = 0.\]
Thus the new morphisms $T_i \to U_i \to X$ patch together, forming a map
$T=\Spec(A)\to X$ through which the initially given 
$\psi:T'\!=\Spec(\bar A)\to X$ factors.
\end{proof}

\section{References}
\frenchspacing
{\ }

[\EGA] A. Grothendieck, J. Dieudonn\'e, {\em \'El\'ements de G\'eom\'etrie
Alg\'ebrique}, Publ. Math. IHES, Part I: 4(1960); Part II: 8(1961); Part III:
11(1961), 17(1963); Part IV: 20(1964), 24(1965), 28(1966), 32(1967).

[\FG] K. Fritzsche, H. Grauert, {\em From Holomorphic Functions to Complex Manifolds},
Springer, 2002. 

[\GW] U. G\"ortz, T. Wedhorn, {\em Algebraic Geometry I}, Vieweg, 2010.

[\Gro] A. Grothendieck, {\em Cat\'egories Cofibre\'es Additives et Complexe Cotangent 
Relatif}, Lect. Notes Math. 79, Springer-Verlag, 1968.

[\Gru] L. Gruson, Dimension homologique des modules plats sur an anneau commutatif 
noeth\'erien, {\em Symposia Mathematica}, Vol. XI (Convegno di Algebra Commutativa,
INDAM, Rome, 1971), Academic Press, London, 1973, 243--254. 

[\Ka] I. Kaplansky, Projective modules, Ann. Math. (2) 68 (1958), 37--377. 

[\MG] J. Majadas, A. Rodicio, {\em Smoothness, Regularity and Complete Intersection}, 
Cambridge University Press, 2010.

[\Ma] M. Maltenfort, {A New Look at Smoothness}, Expo. Math. 20 (2002), 59--93.

[\MO] Answer (by mdeland) to ``Possible formal smoothness mistake in EGA", 2010, 
{\tt http://mathoverflow.net/questions/10731/}

[\RG] M. Raynaud, L. Gruson,  Crit\`eres de platitude e de projectivit\'e, 
Techniques de <<platification>> d'un module, Inv. Math.\ 13, 1--89 (1971). 

[\St] The Stacks Project Authors, {\em Stacks Project}, 2005--2016,
\nl\hbox{\ }{\tt http://stacks.math.columbia.edu}

[\Va] R. Vakil, {\em The Rising Sea: Foundations of Algebraic Geometry}, 2010--2015,
\hbox{\ }{\tt http://math.stanford.edu/\~{}vakil/216blog/index.html}

\end{document}